\documentclass[11pt]{article}

\usepackage{amsmath,amssymb}

\usepackage{times}
\usepackage{bm}
\usepackage{natbib}
\usepackage{algorithm}
\usepackage{xargs}[2008/03/08]
\usepackage{amssymb}
\usepackage{mathrsfs}
\usepackage{graphicx}
\usepackage{rotating,subfigure}
\usepackage[flushleft]{threeparttable}
\usepackage{multirow}
\usepackage{enumitem} 

\usepackage{star}

\newcommand{\nn}{\nonumber}

\usepackage[colorlinks,
linkcolor=blue,
anchorcolor=blue,
citecolor=blue
]{hyperref}

\def\T{{ \mathrm{\scriptscriptstyle T} }}

\def\##1\#{\begin{align}#1\end{align}}
\def\$#1\${\begin{align*}#1\end{align*}}

\def\T{{ \mathrm{\scriptscriptstyle T} }} 

\newcommand{\Rom}[1]{\text{\uppercase\expandafter{\romannumeral #1\relax}}}

\usepackage{geometry}
\usepackage{txfonts}

\newcommand{\scolor}{}

\begin{document}

\title{\LARGE  Gaussian Approximations for Maxima of Random Vectors under $(2+\iota)$-th Moments }

\author{Qiang Sun\thanks{Department of Statistical Sciences, University of Toronto, 100 St. George Street, Toronto, ON M5S 3G3, Canada; E-mail: \texttt{qsun@utstat.toronto.edu}.}}

\date{}

\maketitle


\begin{abstract}
 We derive a Gaussian approximation result for the maximum of a sum of random vectors under $(2+\iota)$-th moments. Our main theorem is abstract and nonasymptotic, and can be applied to a variety of statistical learning problems.  The proof uses the Lindeberg telescopic sum device along with  some other  newly developed technical results.  
 \end{abstract}

\noindent {\bf keywords}
  Gaussian Approximation, Maxima. 

\section{Introduction and Main Result}\label{sec:1}
We derive a Gaussian approximation result for maxima of sums of high dimensional random vectors under $(2+\iota)$-th moments for some $0\leq \iota\leq 1$. This \scolor{complements the results of \cite{chernozhukov2014gaussian} which require third moment condition; see Theorem 4.1 therein.  Later, \cite{chernozhukov2017central} provided high-dimensional central limit and bootstrap theorems for sparsely convex sets.}   Our derivation utilizes the Lindeberg telescopic sum device along with  some other  newly developed technical results.  


 Let $X_1, \ldots, X_n$ be independent random vectors in $\RR^d$ with mean zero and finite $(2+\iota)$-th moments,  that is, $\EE( X_{ij}) =0$ and $\EE\big( |X_{ij}|^{2+\iota}\big) <\infty$, for some $0\leq \iota\leq 1$. Let $\Sigma \equiv\EE \big( X_i X_i^\T\big)$.  Consider the statistic $Z=\max_{1\leq j\leq d}\sum_{i=1}^n X_{ij}.$ Let $Y_1,\ldots, Y_n$ be independent random vectors in $\RR^d$ with $Y_i\sim \cN(0, \Sigma).$ For $0\leq\iota\leq 1$ and  $\gamma, q>0$ such that $\gamma\delta>1$, let
 \scolor{
\#\label{eq:L}
 L_n(\gamma,\delta, \iota)=\min \bigg\{\gamma^2\delta^{-1}\EE  \bigg(\max_{j}\sum \big|X_{ij}\big|^3+\max_{j}\sum \big|Y_{ij}\big|^3\bigg), \gamma^{\frac{4+2\iota}{3}}\delta^{-\frac{2+\iota}{3}}\sum_{i=1}^n C_i(2+\iota)\bigg\},
\#
where
$	C_i(q) = \EE\big(\max_{1\leq j\leq d}|X_{ij}|^{q} + \max_{1\leq j\leq d}|Y_{ij}|^{q}\big).$
} 
\scolor{Let ``$\lesssim$" stand for ``$\leq$" up to a universal constant.} Our main result follows.

\begin{theorem}\label{thm:1}
For any positive scalers $\delta, \gamma$ such that $\delta\gamma>1$ and $\varepsilon=\gamma \delta \exp\{-(\gamma^2\delta^2-1) /2\}<1$, 
   there exists a random variable \scolor{$Z^\dagger {\buildrel d \over =}\max_{1\leq j\leq d}\sum_{i=1}^n Y_{ij}$} such that 
\$
\PP\big(|Z-Z^\dagger |\geq c_\gamma+3\delta\big)\lesssim\frac{\varepsilon+L_n(\gamma, \delta,\iota)}{1-\varepsilon}.
\$
\end{theorem}


\begin{proof}[Proof of Theorem \ref{thm:1}]
The proof of this theorem exploits the smooth approximations for the nonsmooth $\max$ and indicator functions, and the device of Lindeberg's telescopic sum \cite{lindeberg1922neue}.  Because $X_{ij}$'s only have bounded $(2+\iota)$-th moments, the Gaussian comparison inequalities developed previously \citep{chernozhukov2014gaussian} can not be applied, at least not immediately. The key technical difference is Lemma \ref{lemma:lindeberg}, where we uses the device of Lindeberg's telescopic sum. 

The rest of the proof follows from that in \cite{chernozhukov2014gaussian}. We outline it here for completeness. We start  by using a version of Strassen's theorem to prove Theorem \ref{thm:1}, i.e. Lemma 4.1 in \cite{chernozhukov2014gaussian}.  Using this lemma, the conclusion follows immediately if we can prove that for every Borel subset $A$ of $\RR$, 
\#\label{lemma:qa:eq:strassen}
\PP\big(Z\in  A\big)-\PP\big(Z^\dagger \in A^{c_\gamma+3\delta}\big)\leq\frac{\varepsilon+L_n(\gamma, \delta,\iota)}{1-\varepsilon}.
\#

 We shall fix any Borel subset $A$ of $\RR$ throughout the proof.
The first two steps are standard, which involve smooth approximations to the non-smooth maps as discussed previously. We first approximate the non-smooth map $\RR^d\mapsto \RR: x\mapsto \max_{1\leq j\leq d}x_j$ by the smooth function $\psi_\gamma: \RR^d \mapsto \RR$ defined by $\psi_\gamma(x)=\gamma^{-1} \log\big(\sum_{j=1}^d e^{\gamma x_j}\big)$ for $x\in \RR^d$. By elementary calculations, we have for any $x = (x_1,\ldots, x_d)^\T$,
\#\label{eq:smooth}
\max_{1\leq j\leq d} x_j\leq \psi_{\gamma}(x)\leq \max_{1\leq j\leq d} x_j + c_{\gamma}  ,
\#
where $c_{\gamma} = \gamma^{-1}\log d$. 
Similarly, let $S_n = \sum_{i=1}^n X_i$ and $S_n^{\dagger} = \sum_{i=1}^n Y_i$, the Gaussian analogue of $S_n$. Then 
\$
\PP\big(Z\in  A\big)\leq \PP\big( \psi_\gamma(S_n)\in A^{c_{\gamma}}\big) =\EE  \big[ 1_{A^{c_{\gamma}}} \{ \psi_\gamma (S_n ) \}  \big] .
\$

Then we approximate the indicator function $t\mapsto 1_{A}(t)$ by a smooth function. We utilize the following lemma, which is taken from \cite{chernozhukov2014gaussian} and can be traced back to \cite{pollard2002user}. 
\begin{lemma}
\label{lemma:smooth}
Let $\gamma>0$ and $\delta> \gamma^{-1}$. For every Borel subset $A$ of $\RR$, there exists a smooth function $g: \RR\mapsto \RR$ such that $\|g'\|_\infty\leq \delta^{-1}$, $\|g''\|_\infty\leq C \delta^{-1} \gamma$, $\|g'''\|_\infty\leq C \delta^{-1} \gamma^2 $ and
\$
(1-\varepsilon)1_A(t)\leq g(t)\leq \varepsilon+(1-\varepsilon)1_{A^{3\delta}}(t)~\mbox{ for all } t\in \RR,
\$
where $C>0$ is an absolute constant and $\varepsilon=\varepsilon( \gamma,\delta ) = \gamma \delta \exp\{-(\gamma^2\delta^2-1) /2\}<1$. 
\end{lemma}

We take a suitable function $g$ as justified in Lemma \ref{lemma:smooth} to the set $A^{c_\gamma}$ and obtain
\$
\EE  [ 1_{A^{c_\gamma}}\{ \psi_\gamma(S_n) \}  ] \leq (1-\varepsilon)^{-1}\EE  \{ g \circ \psi_{\gamma}(S_n)  \}. 
\$
For simplicity, we write $f = g \circ \psi_\gamma$, i.e., $f(x) = g(\psi_\gamma(x))$ for $x\in \RR$. Then, it suffices to compare $\EE\{ f(S_n)\}$ and $\EE\{ f(S_n^\dagger)\}$ using the smoothness of $f$. If we can establish the following inequality,
\#\label{eq:exp.bd}
\big|\EE f(S_n)- \EE f(S_n^\dagger) \big|\lesssim L_n(\gamma,\delta, \iota),
\#
which is provided in the Lemma \ref{lemma:lindeberg}. 
Then, applying Lemma \ref{lemma:smooth} again, it follows
\$
& \PP\big(Z \in A \big) - \PP\big( Z^\dagger \in A^{c_\gamma+3\delta} \big) \\
 &\leq \EE \big[ 1_{A^{c_\gamma}}\{ \psi_\gamma (S_n ) \} \big] -\PP\big(Z^\dagger \in A^{c_\gamma+3\delta}\big)\leq (1-\varepsilon)^{-1}\EE f(S_n) - \PP\big( Z^\dagger\in A^{c_\gamma+3\delta}	\big)\\
&\lesssim \frac{\EE f(S_n^\dagger) }{1-\varepsilon}-\PP\big(	Z^\dagger \in A^{c_\gamma+3\delta} \big)+\frac{L_n(\gamma, \delta,\iota)}{1-\varepsilon}\leq\frac{\varepsilon+L_n(\gamma, \delta,\iota)}{1-\varepsilon},
\$
where we used the property of the smooth approximation $\psi_\gamma$ in the last inequality. Therefore, we only need to prove \eqref{eq:exp.bd}. This completes the proof. 
\end{proof}

\section{Statement and Proof of Lemma \ref{lemma:lindeberg}}
\begin{lemma}\label{lemma:lindeberg}
Recall the definitions for $f$, $S_n$ and $S_n^\dagger$ in the proof of Lemma \ref{thm:1}. Then, for any $0\leq \iota\leq 1$, we have
\$
|\EE f(S_n)- \EE  f(S_n^\dagger ) | \leq L_n(\gamma, \delta, \iota),
\$
where  \scolor{$L_n(\gamma, \delta, \iota)$ is defined in \eqref{eq:L}.}
\end{lemma}
\begin{proof}[Proof of Lemma \ref{lemma:lindeberg}]
 We use the device of Lindeberg's telescopic sum \citep{lindeberg1922neue}  to prove this lemma. \scolor{Let  $T_i=\sum_{k=1}^{i-1}Y_k+\sum_{k=i}^n X_k$, with $T_1=\sum_{k=1}^nX_k$.} Then, we write $\EE f(S_n) -\EE f(S_n^\dagger)$ as a telescopic sum:
\$
\EE f(S_n) - \EE f(S_n^\dagger) =\sum_{i=1}^n \EE f(T_i) - \EE f(T_{i+1}) .
\$
In order to bound the left-hand side in the above identity, \scolor{we instead bound the telescopic sum.} Let $\Delta_i=T_i-T_{i+1}$ and $L_i=\sum_{k=1}^{i-1}Y_k+\sum_{k=i+1}^nX_k$.  We use $\nabla f$ to denote the derivative, and $\nabla^2 f = (\partial_{jk} f)_{1\leq j, k\leq p}$ the Hessian.  $f(V_i)-f(V_{i+1})$ can be decomposed  as follows:
\#
 f(T_i)-f(T_{i+1}) &=\underbrace{(T_i-T_{i+1})^\T \nabla f(L_i)}_{\Rom{1}_i}\nn\\
 &\quad  + \underbrace{ \tfrac{1}{2} X_i^\T \nabla^2f(L_i) X_i  - \tfrac{1}{2} Y_i^\T \nabla^2f(L_i) Y_i }_{\Rom{2}_i}+R_i,  \label{Ri.definition}
\#
where $R_i$ is the remainder term such that $R_i=f(T_i)-f(T_{i+1})-\Rom{1}_i-\Rom{2}_i$. 

\scolor{Let $R=\sum_{i=1}^n R_i,\  \Rom{1}= \sum_{i=1}^n\Rom{1}_i, \ \text{and}\  \Rom{2}=\sum_{i=1}^n\Rom{2}_i.$ Then $\EE f(S_n) - \EE f(S_n^\dagger) =\EE \Rom{1}+\EE\Rom{2}+\EE{R}$.  In what follows, we bound the expectation of terms $\Rom{1}$, $\Rom{2}$, and $R_i$ respectively.  Starting with $\Rom{1}$, because $T_i - T_{i+1} = X_i - Y_i$, which is independent of $L_i$, we have 
\$
\EE \Rom{1}=  \EE\sum_{i=1}^n \Rom{1}_i=\sum_{i=1}^n \{ \EE(T_i - T_{i+1})  \}^\T   \EE \{ \nabla f(L_i) \} =0.
\$ 
}
\scolor{For $\Rom{2}$, the expectation of $\Rom{2}$ can be bounded by
\$
\EE \Rom{2}&=\EE\left\{\sum_{i=1}^n\sum_{j,k}2^{-1}\partial_{jk}f(L_i)(X_{ij}X_{ik}-Y_{ij}Y_{ik})\right\}\\
&= 2^{-1}\sum_{i=1}^n\sum_{j,k}\EE\big \{\partial_{jk}f(L_i)\big\}  \EE\big\{X_{ij}X_{ik}-Y_{ij}Y_{ik}\big\}   \tag{\text{$X_i,Y_i \perp L_i$}}\\
&=0. \tag{$\EE \{X_{ij}X_{ik}\}=\EE\{Y_{ij}Y_{ik}\}$}
\$}
 In the following lemma, we give an upper bound for the expectation of $R$. 
\scolor{
\begin{lemma}\label{lemma:remainder}
Let $f(x): \RR^d\mapsto \RR$ be defined as in Theorem \ref{thm:1}.  {Then we must have
\$
\EE R &\lesssim L_n(\gamma,\delta, \iota). 
\$}
\end{lemma}
}
\begin{proof}[Proof of Lemma \ref{lemma:remainder}]
Recall the definition of $R=\sum_{i=1}^n R_i$. 
Let $\theta$ be a uniform distributed random variable over $[0,1]$, independent of all other random variables. \scolor{
 Using the third order  Taylor approximation for multivariate functions, we obtain
\$
R_i&=\big\{f(T_i)-f(L_i)\big\}-\big\{f(T_{i+1})-f(L_i)\big\}-\Rom{1}_i-\Rom{2}_i\notag\\
&= {6^{-1}} \EE_\theta \left\{\sum_{j,k,\ell} (1+\theta)^2X_{ij}X_{ik}X_{i\ell}\partial_{jk\ell}  f(L_i+\theta X_{i})\right\} \notag\\
&\qquad\qquad+{6^{-1}}\EE_\theta\left\{\sum_{j,k,\ell} (1+\theta)^2 Y_{ij}Y_{ik}Y_{i\ell} \partial_{jk\ell} f(L_i+\theta Y_{i})\right\} ,
\$
where  the first and second-order terms canceled out. Therefore,  $\EE R$ can be bounded as 
\#\label{eq:lemma.rem:1}
\EE R&= {6^{-1}} \EE \left\{\sum_{i=1}^n\sum_{j,k,\ell} (1+\theta)^2X_{ij}X_{ik}X_{i\ell}\partial_{jk\ell}  f(L_i+\theta X_{i})\right\} \notag\\
&\qquad\qquad+{6^{-1}}\EE\left\{\sum_{i=1}^n\sum_{j,k,\ell} (1+\theta)^2 Y_{ij}Y_{ik}Y_{i\ell} \partial_{jk\ell} f(L_i+\theta Y_{i})\right\} \notag\\ 
&\leq {6^{-1}} \EE \left\{\sum_{j,k,\ell} \|\partial_{jk\ell}  f\|_\infty \max_{j,k,\ell}\sum \big|X_{ij}X_{ik}X_{i\ell}\big|\right\} \notag\\
&\qquad\qquad+{6^{-1}}\EE\left\{\sum_{j,k,\ell} \|\partial_{jk\ell}  f\|_\infty  \max_{j,k,\ell}\sum \big|Y_{ij}Y_{ik}Y_{i\ell}\big|\right\}\notag\\
&= A+B. 
\#
Now we bound $A$ and $B$ respectively. We start with $A$. }
Following elementary calculations along with Lemma \ref{lemma:smooth}, 
we obtain
\begin{gather*}
\sum_{j,k,\ell}^d\big| \partial_{jk\ell}f(x)\big|\leq \|g'''\|_\infty+6\gamma\|g''\|_\infty+6\gamma^2\|g'\|_\infty\leq   (7C+6)\gamma^2\delta^{-1} \lesssim \gamma^2\delta^{-1}.
\end{gather*}
which, combined with equation \eqref{eq:lemma.rem:1}, yields  \scolor{
\#
A  &\leq \frac{1}{6}(7C+6) \gamma^2\delta^{-1} \EE \left\{ \max_{j,k,\ell}\sum \big|X_{ij}X_{ik}X_{i\ell}\big|\right\} \nn\\
&\lesssim \gamma^2\delta^{-1} \EE \left\{ \max_{j,k,\ell}\sum \big|X_{ij}X_{ik}X_{i\ell}\big|\right\}\lesssim  \gamma^2\delta^{-1} \EE \left\{ \max_{j}\sum \big|X_{ij}\big|^3\right\}.  \label{Ri.bound2}
\#
Similarly,
\#
B  &\lesssim  \gamma^2\delta^{-1} \EE \left\{ \max_{j}\sum \big|Y_{ij}\big|^3\right\}.  \label{Ri.bound3}
\#
}
\scolor{Now using the fact that $0\leq f(x)\leq 1$ and $\EE \Rom{1}=\EE\Rom{2}=0$, we obtain
\#\label{Ri.bound1}
\EE R &=  \EE f(S_n) - \EE f(S_n^\dagger) -\EE \Rom{1}-\EE\Rom{2}\leq 1.
\#}
\scolor{
Putting the upper bounds \eqref{Ri.bound2}, \eqref{Ri.bound3}, and \eqref{Ri.bound1} together yields 
\$
\EE R &\lesssim \min \left\{1, \gamma^2\delta^{-1}\EE  \bigg\{\max_{j}\sum \big|X_{ij}\big|^3\bigg\}+ \gamma^2\delta^{-1}\EE  \bigg\{\max_{j}\sum \big|Y_{ij}\big|^3\bigg\}\right\}. 
\$
Using the fact that $R_i=f(T_i)-f(T_{i+1})-\Rom{1}_i-\Rom{2}_i$ and in a similar argument, we shall obtain 
\$
\EE R_i &\lesssim  \gamma^{-1}\delta^{-1}\EE\min\Big\{\gamma\delta +  \gamma\Big(\max_{1\leq j\leq d}|X_{ij}| + \max_{1\leq j\leq d}|Y_{ij}|\Big) \\
&\qquad+ \gamma^2\Big(\max_{1\leq j\leq  d}|X_{ij}|^2 + \max_{1\leq j\leq d}|Y_{ij}|^2\Big),   \gamma^3\Big(\max_{1\leq j\leq d}|X_{ij}|^3 + \max_{1\leq j\leq d}|Y_{ij}|^3\Big) \Big\}.
\$
}
We need the following lemma, which enables the relaxation of the moment conditions. 
\begin{lemma}\label{lemma:moment}
Let $a\geq 1$ and $x\geq 0$. For any $0 \leq \iota \leq  1$, we have 
\begin{gather*}
\min\big\{a+x+x^2, x^3\big\} \leq 3 a^{(1-\iota)/3} x^{2+\iota}.
\end{gather*}
\end{lemma}

\begin{proof}[Proof of Lemma \ref{lemma:moment}]
Using the fact that $a>1$ and splitting the support of $x$,  we obtain 
\$
\min\big\{a+x+x^2, x^3\big\}&\leq 3\min\big\{a\vee x\vee x^2, x^3 \big\}\\
&\leq 3 \Big(\min\big\{a, x^3\big\}1\big(x\leq 1\big)+ \min\big\{a, x^3\big\}1\big(1<x\leq a^{1/3}\big)\\
&~~~~+\min\big\{a, x^3\big\}1\big(a^{1/3}<x\leq  a^{1/2}\big)+\min\big\{x^2, x^3\big\}1\big(x>a^{1/2}\big)\Big)\\
&\leq 3   a^{(1-\iota)/3} x^{2+\iota}.
\$ 
\end{proof}
Applying Lemma \ref{lemma:moment} with { $x=\gamma\max(|X_{ij}|,|Y_{ij}|)$}, we obtain \scolor{
\$
\EE R_i\lesssim \gamma^{(4+2\iota)/3}\delta^{-(2+\iota)/3}C_i (2 + \iota),
\$
}
where $C_i(2\!+\!\iota) = \EE\big(\max_{1\leq j\leq d}|X_{ij}|^{2+\iota} + \max_{1\leq j\leq d}|Y_{ij}|^{2+\iota}\big)$. Combining two different bounds together yields  Lemma \ref{lemma:remainder}. 
\end{proof}

\end{proof}

\bibliographystyle{ims}
\bibliography{SQ2016_RGA}

%
%

\end{document}